\documentclass[a4paper]{article}
\usepackage[utf8]{inputenc}
\usepackage{amsthm,amsfonts,amsmath,verbatim,amssymb,verbatim,cite}
\usepackage{graphicx,cleveref}
\usepackage[usenames]{color}

\numberwithin{equation}{section}
\newtheorem{theorem}{Theorem}[section]

\newtheorem{conj}[theorem]{Conjecture}

\newtheorem{problem}[theorem]{Problem}

\newtheorem{remark}[theorem]{Remark}

\title{Decomposition of cubic graphs related to Wegner's conjecture}
\author{J\'anos Bar\'at\thanks{Supported by Sz\'echenyi 2020 under the
    EFOP-3.6.1-16-2016-00015 and 
OTKA-ARRS Slovenian-Hungarian Joint Research Project, grant no. NN-114614.}\\
\small University of Pannonia, Department of Mathematics\\[-0.8ex]
\small 8200 Veszpr\'em, Egyetem utca 10., Hungary\\[-0.8ex]
\small and\\
\small  MTA-ELTE Geometric and Algebraic Combinatorics Research Group\\[-0.8ex]
\small H--1117 Budapest, P\'azm\'any P.\ s\'et\'any 1/C, Hungary\\[-0.8ex]
\small \texttt{barat@cs.elte.hu}}

\begin{document}

\maketitle

\begin{abstract}
Thomassen formulated the following conjecture: 
Every $3$-connected cubic graph has a red-blue vertex coloring such that
the blue subgraph has maximum degree $1$
(that is, it consists of a matching and some isolated vertices) and the red
subgraph has minimum degree at least $1$ and contains no $3$-edge path.
We prove the conjecture for Generalized Petersen graphs.

We indicate that a coloring with the same properties might exist for any subcubic graph.
We confirm this statement for all subcubic trees.
\end{abstract}

\section{Preliminaries}

We use standard graph theory terminology.
For any undefined terms, consult the excellent book by Jensen and Toft \cite{JTbook}.
The {\em square chromatic number} of a graph $G$ is simply the chromatic number of the square of $G$.
In notation, $\chi_{\square}(G)=\chi(G^2)$.
A graph is {\it cubic} if it is 3-regular.
A graph is {\it subcubic} if it has maximum degree 3.

Wegner \cite{weg} initiated the study of the square chromatic number of planar graphs.
There has been accelerated interest in this topic due to his conjecture.
We recall the case $\Delta\le 3$.

\begin{conj}[Wegner \cite{weg}]
 For any subcubic planar graph $G$, the square of $G$ is $7$-colorable. That is, $\chi_{\square}(G)\le 7$.
\end{conj}

Recently, Thomassen published his proof of Wegner's conjecture, see \cite{ct17}.
He formulated an attractive conjecture, which would imply Wegner's. 
This new conjecture belongs to another well-studied area of graph theory: graph decompositions.
It is well-known, that subcubic graphs can be 2-colored such that each color class induces only a matching and isolated vertices.
To see this, one distributes the vertices arbitrarily into a red and a blue class. 
If inside any class, there is a vertex of induced degree at least 2, then swap the color of that vertex.
At first sight, the next conjecture is very similar, but excludes the possibility of isolated vertices in the red class.
Instead it relaxes the red part from a matching to a subgraph not containing a $3$-edge path. 

\begin{conj}[Thomassen] \label{r&b} 
 If $G$ is a $3$-connected, cubic graph on at least $8$ vertices, then the vertices of $G$ can
be colored blue and red such that the blue subgraph has maximum degree $1$
(that is, it consists of a matching and some isolated vertices) and the red
subgraph has minimum degree at least $1$ and contains no $4$-path.
\end{conj}

Here a $4$-path is a path with $3$ edges and 4 vertices.

\begin{remark}
 The original conjecture was formulated for all $3$-connected cubic graphs.
 The reviewer of this note observed that the $3$-prism does not have the required red-blue coloring. 
 Therefore, it has to be excluded with an extra assumption.
\end{remark}

Thomassen gave a short and elegant argument that shows how Conjecture~\ref{r&b} implies Wegner's conjecture.
In this note, we confirm Conjecture~\ref{r&b} for Generalized Petersen graphs.
In the context of the square chromatic number of subcubic graphs, 
the Petersen graph plays a crucial role in the following sense. 
We know that  $\chi_{\square}=10$ for the Petersen graph.
Cranston and Kim~\cite{cra} proved that a dramatic drop happens in the chromatic number if we exclude only this one graph:
$\chi_{\square}(G)\le 8$ for any subcubic graph $G$ different from the Petersen graph.
However, Conjecture~\ref{r&b} holds also for the Petersen graph, as shown in Figure~\ref{petersen}.

\begin{figure}[h]
\minipage{0.49\textwidth}
  \includegraphics[width=\linewidth]{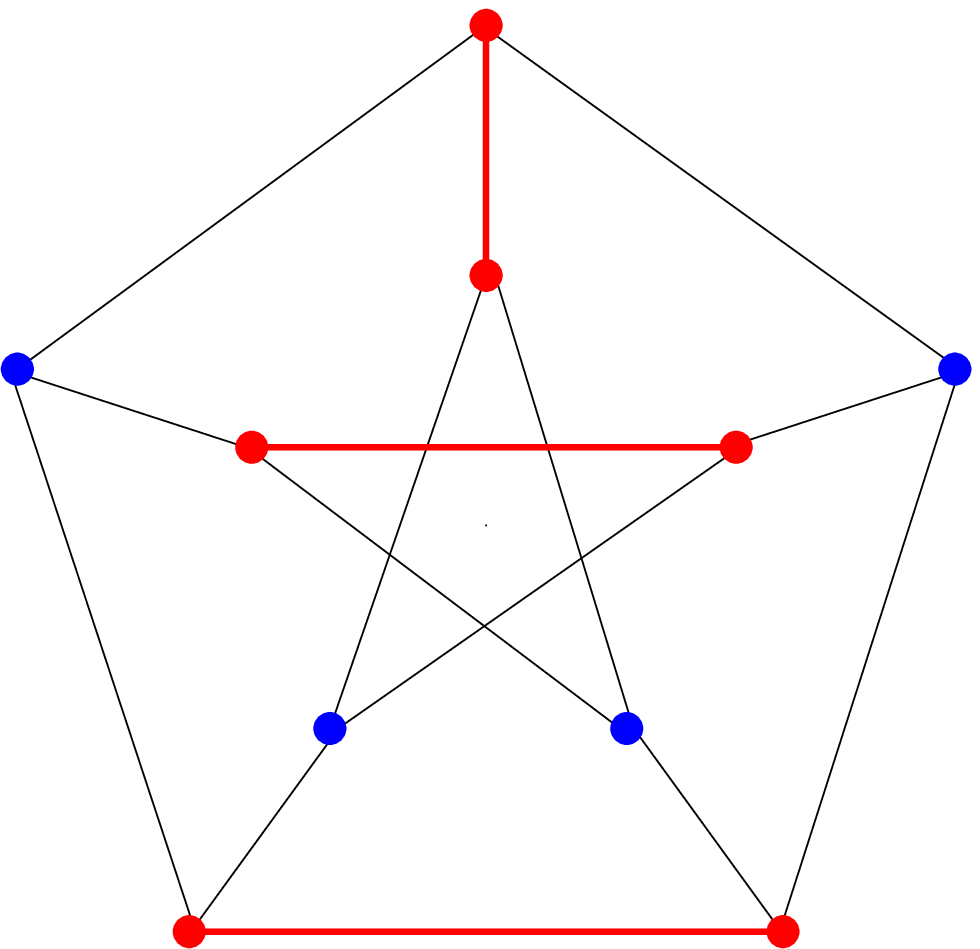}
  \caption{A red-blue coloring of the Petersen graph satisfying Conjecture~\ref{r&b}.}
\label{petersen}
\endminipage\hfill
\minipage{0.49\textwidth}
\includegraphics[width=\linewidth]{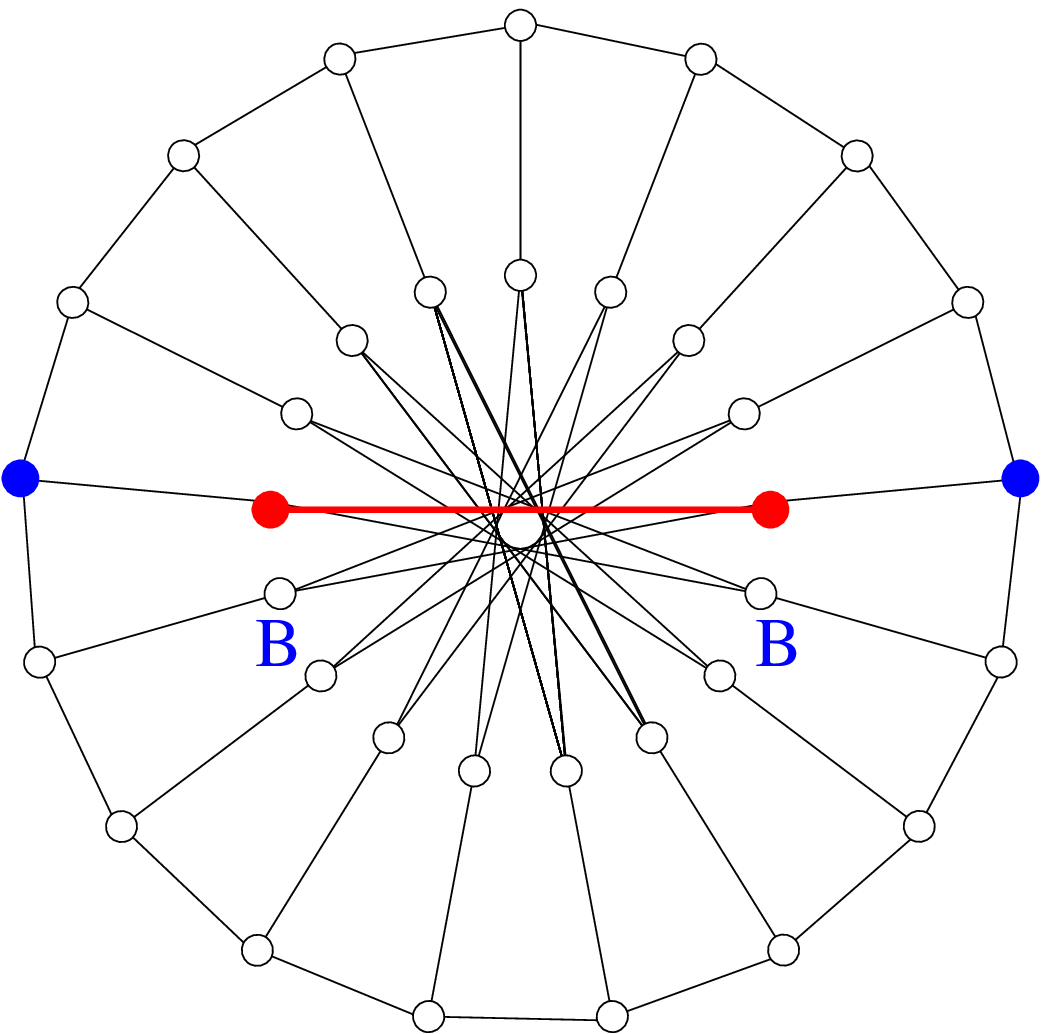}
  \caption{A horizontal cut.}
\label{horizontal}
\endminipage
\end{figure}

The Generalized Petersen graph $GP(2k+1,k)$ is defined for $k\ge 2$ as follows: 
the vertices are $\{u_1,\dots,u_{2k+1}\}$ and  $\{v_1,\dots,v_{2k+1}\}$ such that 
$\{u_1,\dots,u_{2k+1}\}$ form a cycle (the outer cycle) in the natural order. 
The spoke edges are of form $u_iv_i$ for $1\le i\le 2k+1$.
Finally, the inner cycle is spanned by the long diagonals. 
That is, edges of form $v_iv_{i+k}$, where the indices are modulo $2k+1$ and $1\le i\le 2k+1$.

One might define somewhat similar objects for an even number as follows: for $l\ge 3$ let $GP(2l,l-1)$ 
have vertices $\{u_1,\dots,u_{2l}\}$ and  $\{v_1,\dots,v_{2l}\}$ such that 
$\{u_1,\dots,u_{2l}\}$ form the outer cycle in the natural order. 
The spoke edges are of form $u_iv_i$ for $1\le i\le 2l$.
Finally, the inner 2-factor is spanned by the second longest diagonals. 
That is, edges of form $v_iv_{i+l-1}$, where the indices are modulo $2l$ and $1\le i\le 2l$.
However, these edges form a cycle only if $2l$ and $l-1$ are coprime. 
Therefore, $l$ needs to be even.
Hence we use $l=2k$ in the even part of the next section.

\section{Generalized Petersen graphs}

For the odd case, we can prove the following 

\begin{theorem}
 For any $k\ge 2$, {\rm Conjecture}~$\ref{r&b}$ holds for the Generalized Petersen graph $GP(2k+1,k)$.
 That is, there exists a red-blue vertex coloring such that the induced blue components are vertices or edges and the
 red components are stars with $1,2$ or $3$ edges.
\end{theorem}

\begin{proof}
 We are going to construct an explicit coloring.
There are three cases according to the number of vertices on the outer cycle modulo 6.
 
We use the following building blocks: horizontal cut, red wedge, blue cross and red syringe, shown in 
Figures~\ref{horizontal}-\ref{syringe}.

A horizontal cut corresponds to vertices $u_1,u_k$ and $v_1,v_k$, where $u_1,u_k$ are blue and $v_1,v_k$ are red.

A red wedge corresponds to vertices $u_i,v_i,u_{i+k-1},u_{i+k},v_{i+k-1},v_{i+k}$ for any $i$, where
the vertices $u_i,u_{i+k-1},u_{i+k}$ are blue and  $v_i,v_{i+k-1},v_{i+k}$ are red.

A blue cross corresponds to vertices $(u_{i-1},u_i,u_{i+1})$, $(v_{i-1},v_i,v_{i+1})$, $(u_{i+k},u_{i+k+1})$, $(v_{i+k},v_{i+k+1})$ for any $i$,
where $(u_{i-1},u_i,u_{i+1},v_i),(u_{i+k},u_{i+k+1})$ are red and $(v_{i-1},v_{i+k}),(v_{i+1},v_{i+k+1})$ are blue.
 
A red syringe corresponds to vertices $u_i,v_i,u_{i+k-1},u_{i+k},v_{i+k-1},v_{i+k}$ for any $i$, where
the vertices $u_i,v_i,u_{i+k-1},u_{i+k}$ are red and $v_{i+k-1},v_{i+k}$ are blue.
 
 \begin{figure}[h]
\minipage{0.33\textwidth}
\includegraphics[width=\linewidth]{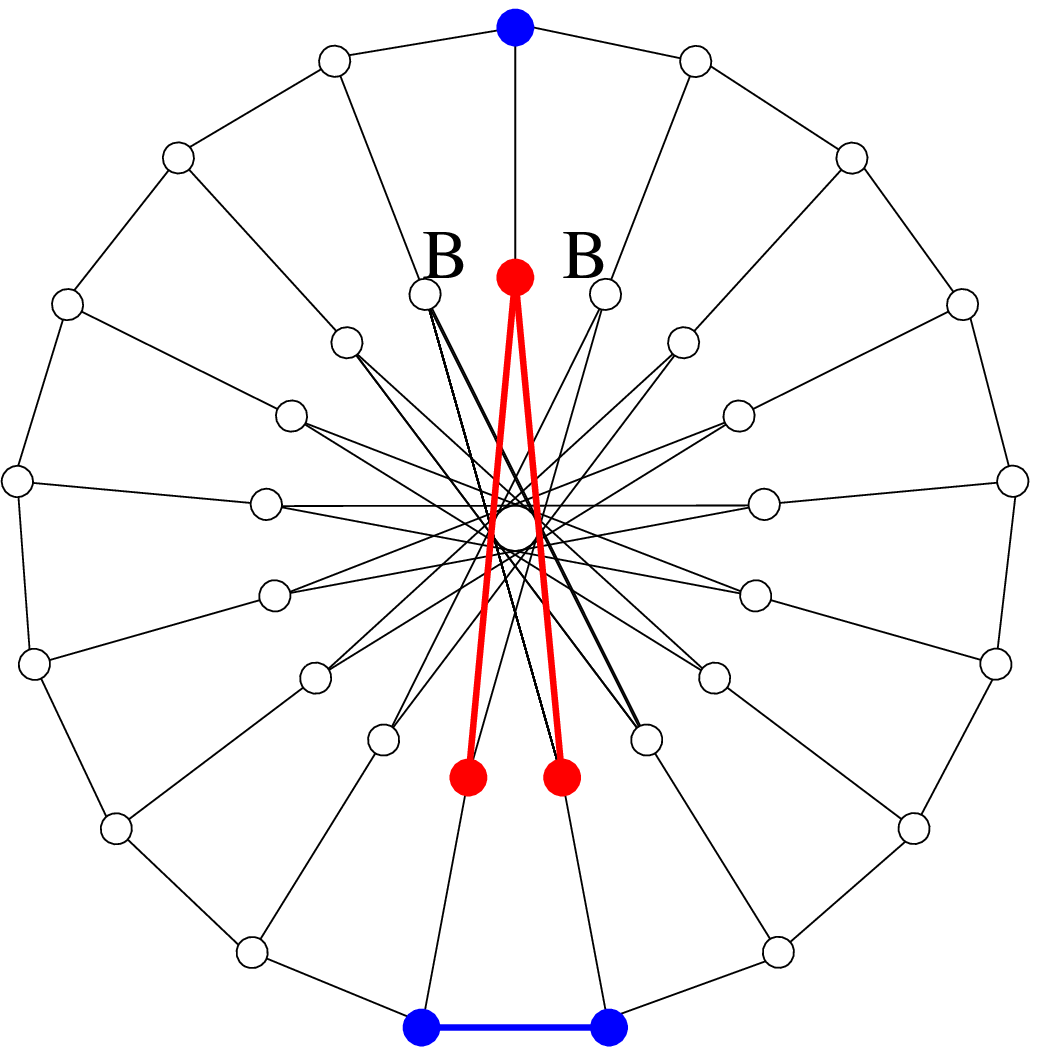}
  \caption{A red wedge.}
\label{redwedge}
\endminipage\hfill
\minipage{0.33\textwidth}
  \includegraphics[width=\linewidth]{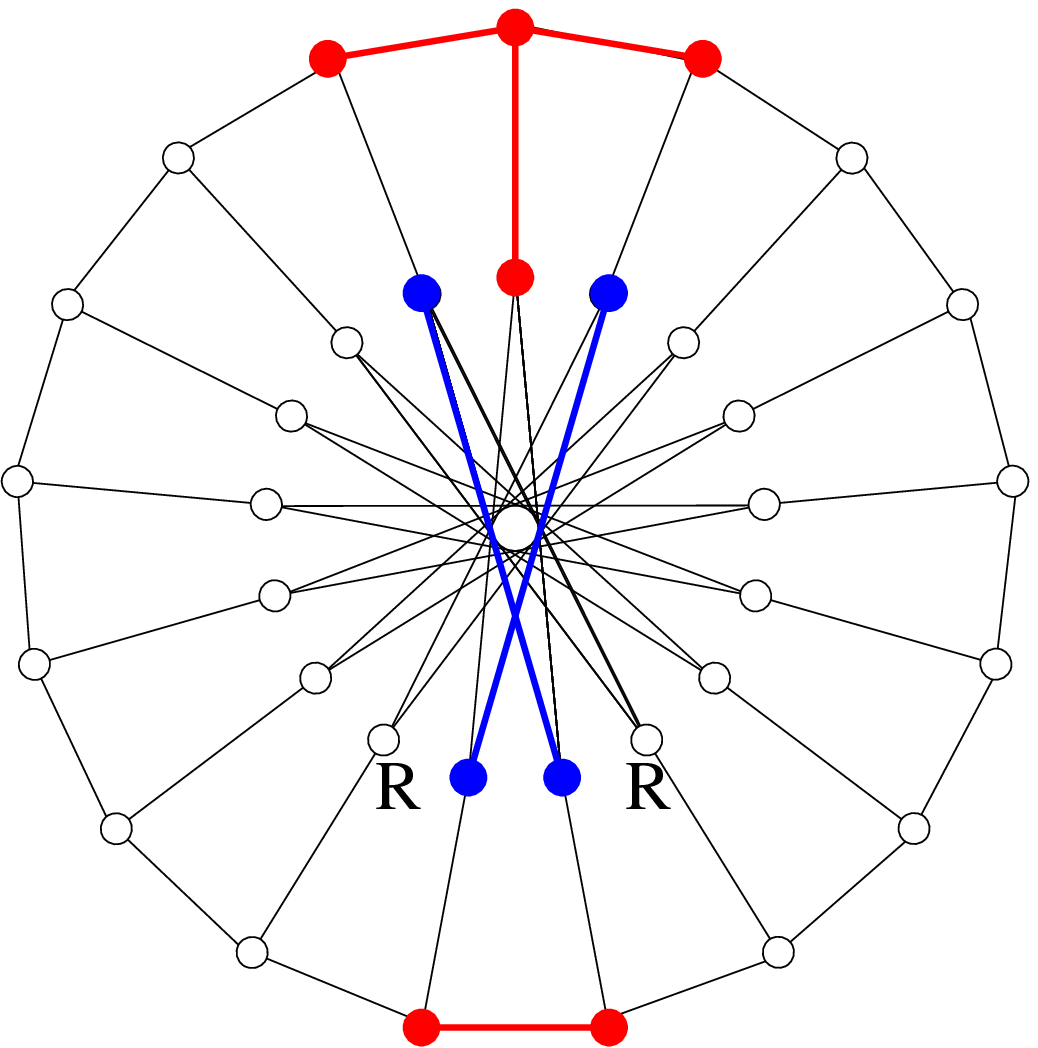}
  \caption{A blue cross.}
\label{bluecross}
\endminipage\hfill
\minipage{0.33\textwidth}
  \includegraphics[width=\linewidth]{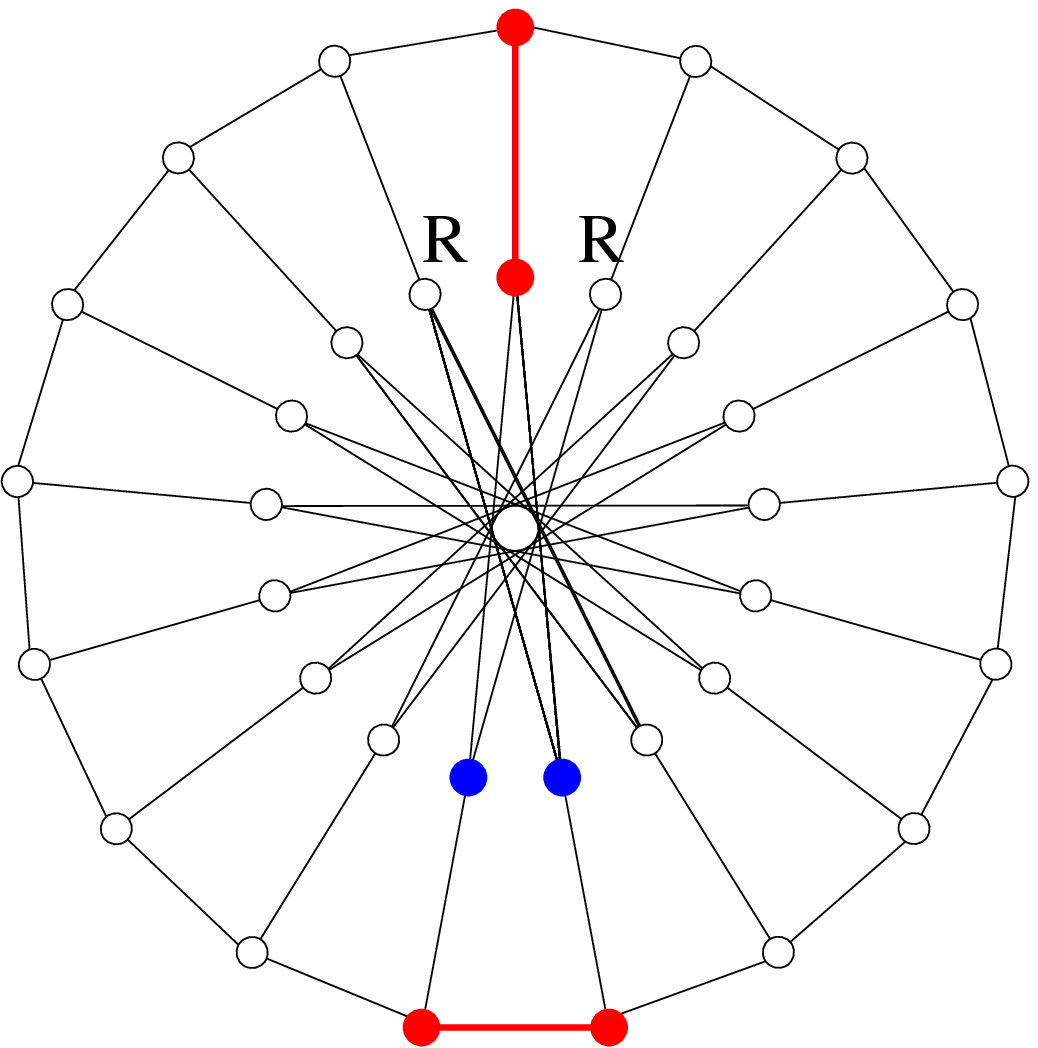}
  \caption{A red syringe.}
\label{syringe}
\endminipage
\end{figure}

 Each of these colored vertex sets cut the outer cycle into two disjoint paths $P$ and $Q$ and  
 divide the Generalized Petersen graphs into two halves in the figures: 
 the vertices of $P$ and their neighbors form one half, the vertices of $Q$ and their neighbors the other half. 
 Apart from the colored vertices, a building block might impose a side condition.
 That reflects the imposed color of some vertices in the inner cycle, 
 see the vertices marked by R or B in Figures~\ref{horizontal}-\ref{syringe}.
 It means that our construction will obey the imposed colors.
 (Many of the colors are necessary by the definition of the red-blue coloring.)
 We describe our construction by listing the colors of the vertices on the outer cycle.
 The horizontal cut is used first and precisely once in each case.
 After that, we are given an upper and a lower subgraph. 
 We list the color sequence on the outer cycle of the upper subgraph in the clockwise order, 
 then the color sequence on the outer cycle of the lower subgraph in the clockwise order.
 The two sequences are separated by a comma.
 For instance, 
 {\color{red}R}{\color{blue}B}{\color{blue}B}{\color{red}R}{\color{red}R}{\color{blue}B}{\color{blue}B}{\color{red}R}{\color{red}R}{\color{blue}B}{\color{blue}B}{\color{red}R},{\color{red}R}{\color{red}R}{\color{blue}B}{\color{red}R}{\color{red}R}{\color{red}R}{\color{blue}B}{\color{red}R}{\color{red}R}{\color{red}R}{\color{blue}B}{\color{red}R}{\color{red}R}, see Figure~\ref{color_GP27}.
 One decodes this as follows: syringe, wedge, cross, wedge, cross, wedge, syringe. 
 In each decoding step, recognising a building block, we consider the next monochromatic sequence in the color sequence (from left to right) before and after the comma.
The horizontal cut is not present in the color sequence.

In the inductive steps below, we will assume the existence of a coloring of a Generalized Petersen graph, which is represented by
the coloring of the upper and lower path of the outer cycle, C,D say.
This shorthand is used for an appropiate series of {\color{red}R}s and {\color{blue}B}s such that a comma is used as explained above.
After that, we consider a one size larger Generalized Petersen graph $G$ and extend the coloring C,D by some appropriate building
blocks to get the required coloring of $G$.

\smallskip

\underline{Case 1.} Let the number of vertices on the outer cycle be $5+6s$. 
That is, we color the graphs $GP(5+6s,2+3s)$, where $s\ge 0$.
The initial sequence (when $s=0$) is {\color{red}R},{\color{red}R}{\color{red}R}, which corresponds to the Petersen graph, see Figure~\ref{petersen}.
Now assume that the coloring C,D of  $GP(5+6s,2+3s)$ is given. 
We extend this coloring by adding a syringe and a wedge (such that the number of vertices increases by 3 both in the upper and the lower part) right after the horizontal cut.
That is, {\color{red}R}{\color{blue}B}{\color{blue}B}C,{\color{red}R}{\color{red}R}{\color{blue}B}D is the coloring of $GP(5+6(s+1),2+3(s+1))$.

\smallskip

\underline{Case 2.} Let the number of vertices on the outer cycle be $7+6s$.
That is, we color the graphs $GP(7+6s,3+3s)$, where $s\ge 0$.
The initial sequence (when $s=0$) is {\color{red}R}{\color{red}R},{\color{red}R}{\color{red}R}{\color{red}R} corresponding to a blue cross.
Now assume that the coloring C,D of  $GP(7+6s,3+3s)$ is given. 
We extend this coloring by adding a syringe and a wedge right after the horizontal cut.
That is, {\color{red}R}{\color{blue}B}{\color{blue}B}C,{\color{red}R}{\color{red}R}{\color{blue}B}D is the coloring of $GP(7+6(s+1),3+3(s+1))$.

\smallskip

\underline{Case 3.} Let the number of vertices on the outer cycle be $15+6s$.
That is, we color the graphs $GP(9+6s,4+3s)$, where $s\ge 1$.
The initial sequence is {\color{red}R}{\color{red}R}{\color{blue}B}{\color{blue}B}{\color{red}R}{\color{red}R},{\color{red}R}{\color{red}R}{\color{red}R}{\color{blue}B}{\color{red}R}{\color{red}R}{\color{red}R} 
corresponding to cross, wedge, cross.
Now assume that the coloring {\color{red}R}{\color{red}R}C,{\color{red}R}{\color{red}R}{\color{red}R}D of  $GP(9+6s,4+3s)$ is given. 
We extend this coloring by adding a wedge and a syringe after the first cross.
That is, {\color{red}R}{\color{red}R}{\color{blue}B}{\color{blue}B}{\color{red}R}C,{\color{red}R}{\color{red}R}{\color{red}R}{\color{blue}B}{\color{red}R{\color{red}R}}D is the coloring of $GP(9+6(s+1),4+3(s+1))$.

The exceptional case $GP(9,4)$ can be done as depicted in Figure~\ref{color_GP9}.

\begin{figure}[!h]
\minipage{0.39\textwidth}
  \includegraphics[width=\linewidth]{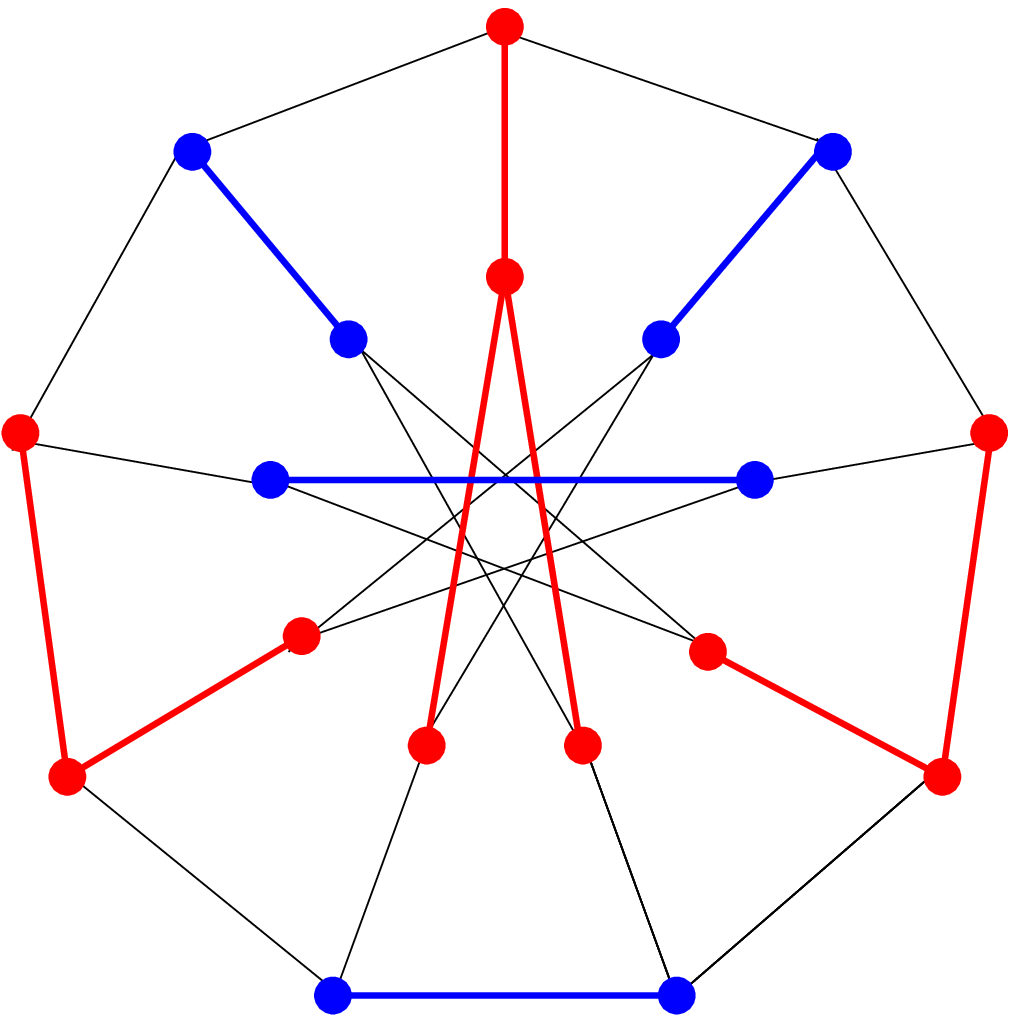}
  \caption{A coloring of $GP(9,4)$ satisfying the conditions of Conjecture~\ref{r&b}.}
\label{color_GP9}
\endminipage\hfill
\minipage{0.59\textwidth}
  \includegraphics[width=\linewidth]{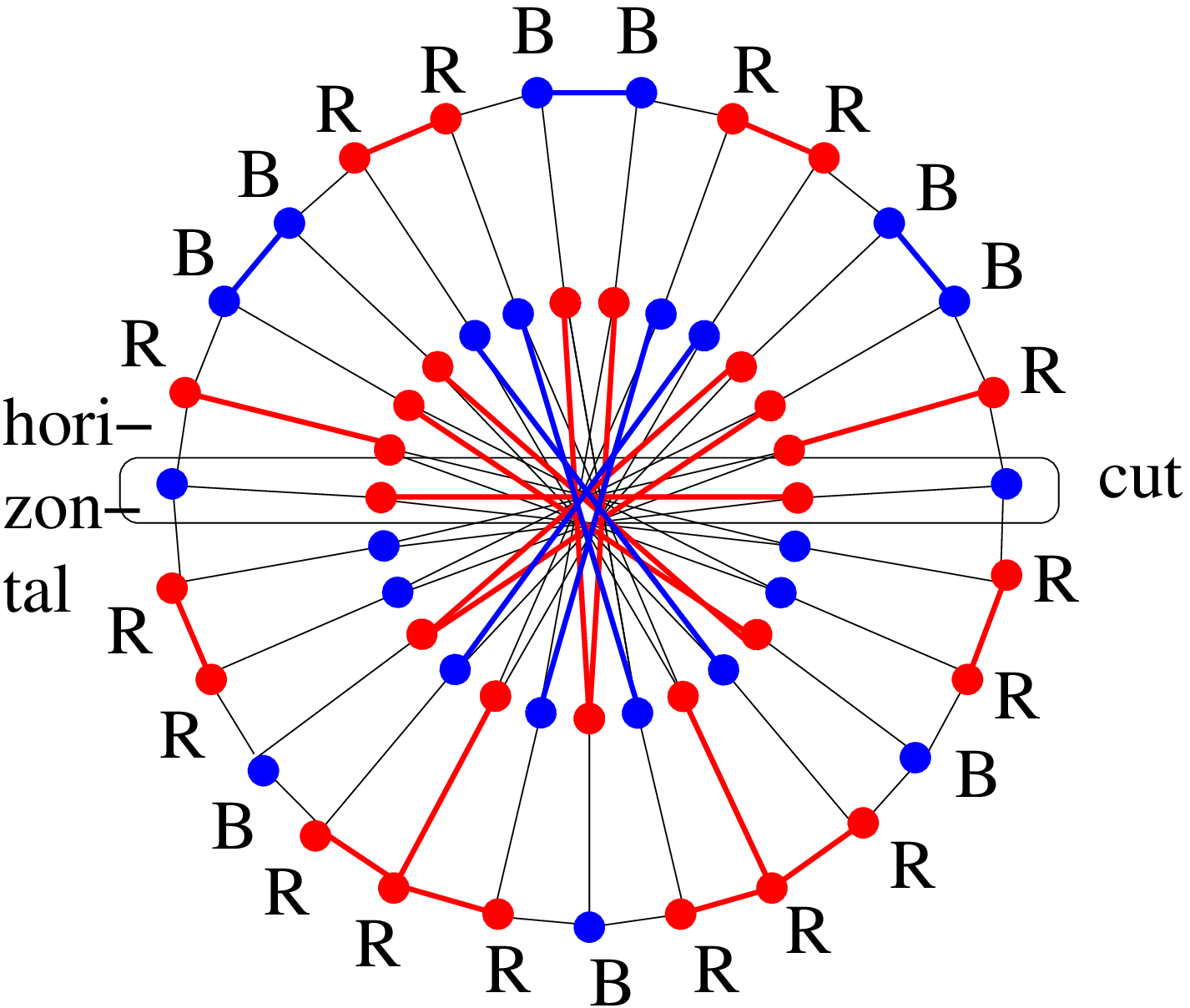}
  \caption{A coloring of $GP(27,13)$ satisfying the conditions of Conjecture~\ref{r&b}.}
\label{color_GP27}
\endminipage
\end{figure}

After giving the recipe in our construction, we only have to check that the side conditions are satisfied.
That means, the building blocks can be put together nicely after each other.

We always start and end the coloring sequence of the upper subgraph by a cross or a syringe. 
Therefore, the side condition of the horizontal cut is always satisfied, 
since a cross or a syringe have blue vertices on the inner cycle, where the side condition of the horizontal cut requires that.

The side condition of a wedge is always satisfied, since we can only put cross or syringe next to a wedge.
However, both a cross and an appropriately turned syringe contains blue vertices on the inner cycle, where the side condition
of the wedge requires that.

Reversing the argument of the previous paragraph, we can similarly confirm that the side condition of a syringe or a cross is always satisfied.
Indeed, we only put a wedge next to a syringe or a cross, and a wedge has only red vertices on the inner cycle.
\end{proof}

For the even case, we prove the following

\begin{theorem}
 For any $k\ge 2$, {\rm Conjecture}~$\ref{r&b}$ holds for the Generalized Petersen graph $GP(4k,2k-1)$.
 That is, there exists a red-blue vertex coloring such that the induced blue components are vertices or edges and the
 red components are stars with $1,2$ or $3$ edges.
\end{theorem}

\begin{proof}
 We consider two cases according to the parity of $k$. 
 It is easy to find the required coloring, when $k$ is even. 
 In that case, the number of vertices on the outer cycle is divisible by 8.
 Hence we can repeatedly color the vertices  {\color{red}R}{\color{red}R}{\color{blue}B}{\color{blue}B}\dots, so a colored matching is induced.
 Now consider two colored ``diagonally opposite'' edges on the outer cycle, $u_{1}u_{2}$ and $u_{2k+1}u_{2k+2}$ say.
 Suppose these vertices are red.
 Consider the vertices of the inner cycle next to these 4 vertices: $v_{1},v_{2}$ and $v_{2k+1},v_{2k+2}$.
 We color them blue and they induce a matching $v_1v_{2k+2}$ and $v_2v_{2k+1}$ on the inner cycle, since $(2k+2)+(2k-1)\equiv 1 \text{\ mod\ }4k$.
 We repeat this idea on the rest of the graph, adding 2 to all indices and interchanging the colors.
 This coloring satisfies the conditions.
 
  \begin{figure}[!h]
\begin{center}
  \includegraphics[width=0.4\linewidth]{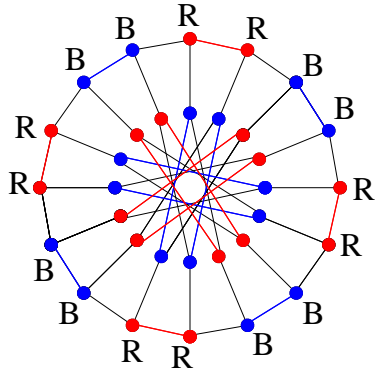}
  \caption{A coloring of $GP(16,7)$ satisfying Conjecture~\ref{r&b}.}
\label{even16_color}
\end{center}
\end{figure}
 
 When $k$ is odd, we construct the coloring by giving an initial coloring and then adding the colored matchings as in the 
 previous argument.
 Assume that $k\ge 5$.
 The initial graph is $GP(20,9)$.
 We use the coloring given in Figure~\ref{color_GP20}.
 As before, we can identify the coloring by listing the color sequence of the vertices on the outer cycle in clockwise order
 starting at $u_1$:
 {\color{red}R}{\color{red}R}{\color{red}R}{\color{blue}B}{\color{blue}B}{\color{red}R}{\color{red}R}{\color{red}R}{\color{blue}B}{\color{blue}B},{\color{red}R}{\color{red}R}{\color{red}R}{\color{blue}B}{\color{blue}B}{\color{red}R}{\color{red}R}{\color{red}R}{\color{blue}B}{\color{blue}B}.
 Here the decoding is the following: three red vertices correspond to a red anchor, as indicated in Figure~\ref{color_GP20}.
 For instance a red anchor represented by {\color{red}R}{\color{red}R}{\color{red}R} include the 4 vertices $u_1,u_2,u_3,v_2$ belonging to a red claw, and its two blue neighbors $v_1,v_3$ on the inner cycle.
 Here the magic works, since the edges of the inner cycle connect vertices of opposite colors.
 Two consecutive blue vertices {\color{blue}B}{\color{blue}B} on the outer cycle correspond to a blue edge, for instance $u_4,u_5$ and its two red neighbors $v_4,v_5$ on the inner cycle.

 \begin{figure}[!h]
\minipage{0.49\textwidth}
  \includegraphics[width=\linewidth]{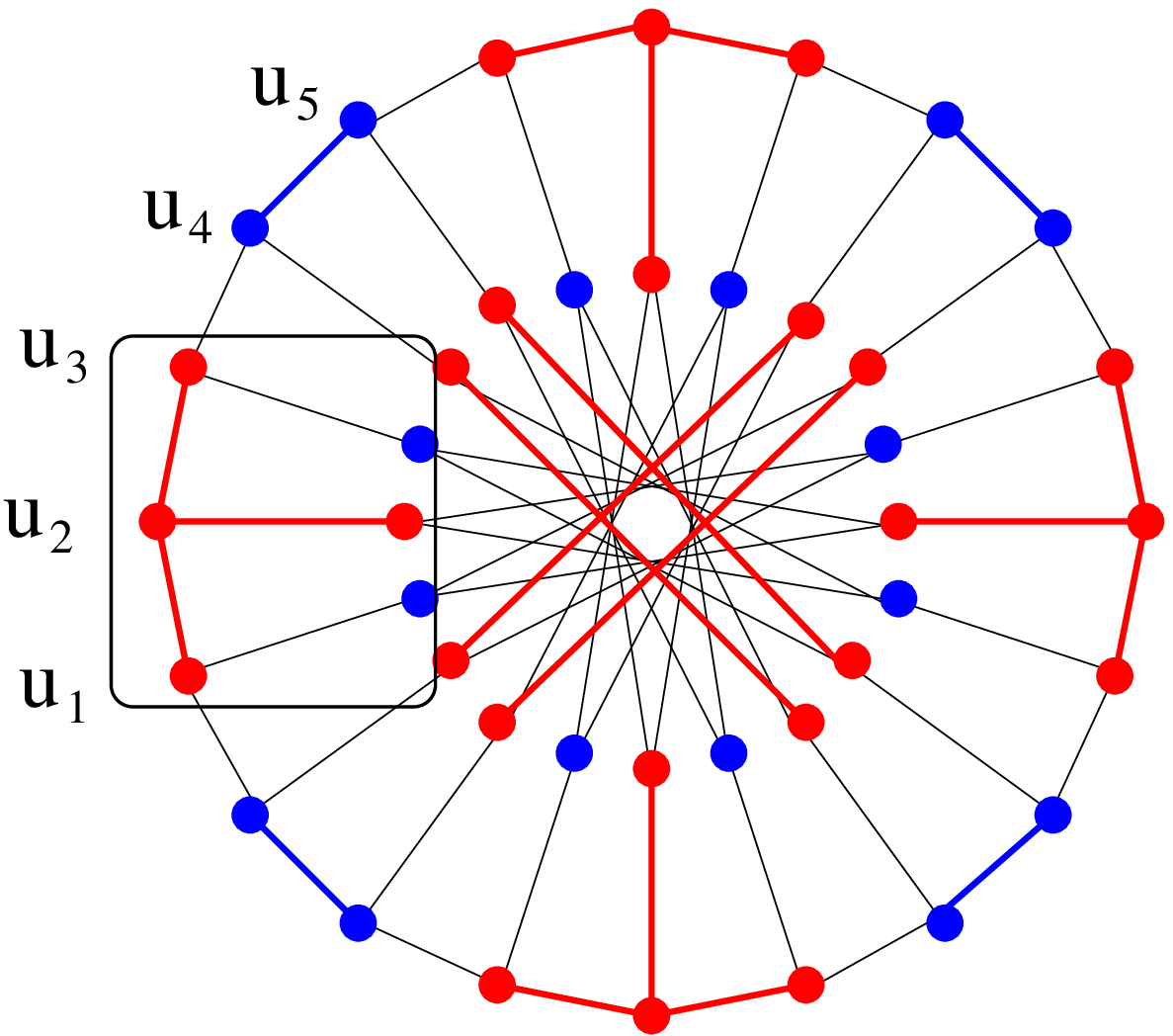}
  \caption{A coloring of $GP(20,9)$ satisfying Conjecture~\ref{r&b}.}
\label{color_GP20}
\endminipage\hfill
\minipage{0.49\textwidth}
  \includegraphics[width=\linewidth]{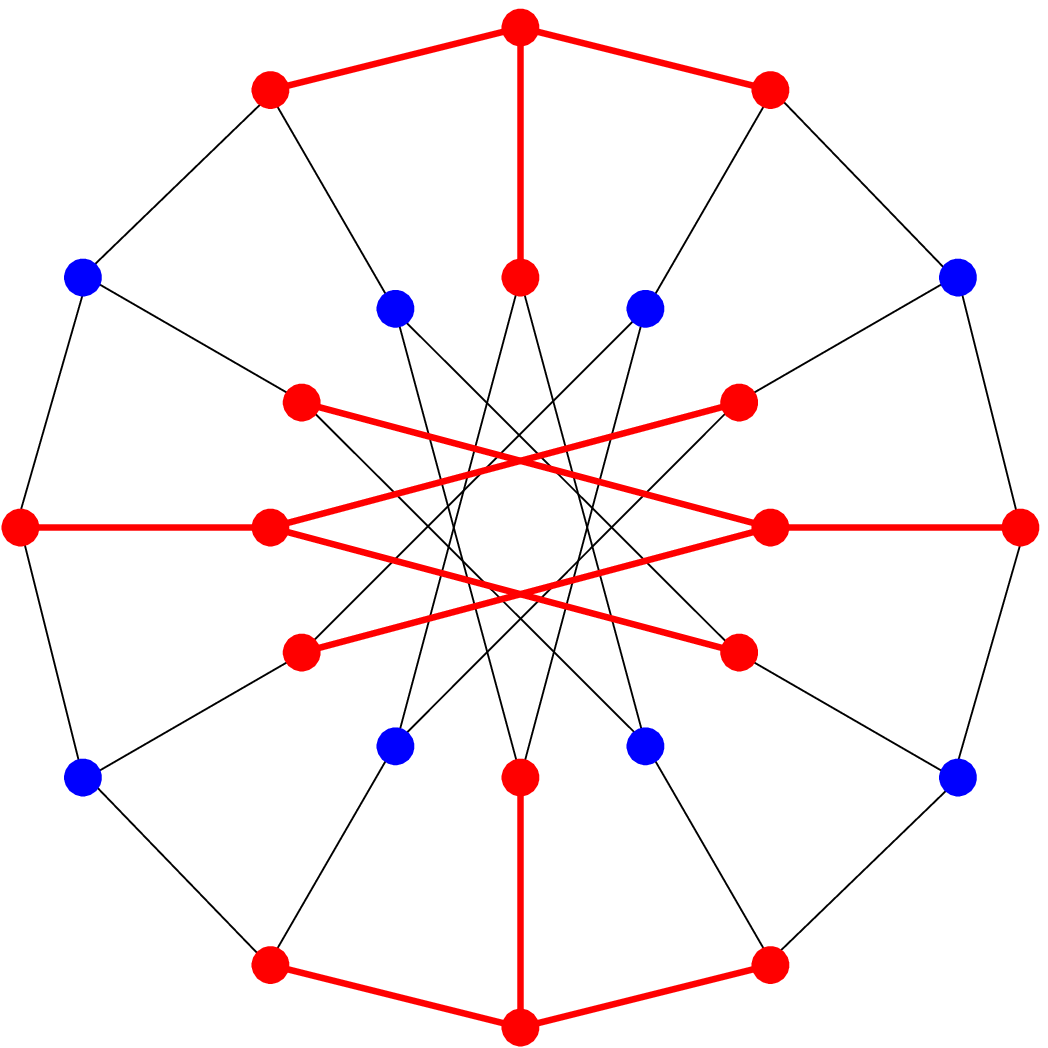}
  \caption{A coloring of $GP(12,5)$ satisfying Conjecture~\ref{r&b}.}
\label{color_GP12}
\endminipage
\end{figure}

When $k$ increases by two, the number of vertices on the outer cycle increases by 8.
Therefore, the following extension works:
We add two red vertices and two blue vertices at the end of the color sequence.
Formally, if C,C was a valid coloring of $GP(4k,2k-1)$, where $k$ is odd, then
C{\color{red}R}{\color{red}R}{\color{blue}B}{\color{blue}B},C{\color{red}R}{\color{red}R}{\color{blue}B}{\color{blue}B} is a valid coloring of $GP(4(k+2),2k+3)$.
This extension corresponds to inserting two colored edges on the outer cycle and two pairs of colored edges of the inner cycle. 
This is the same principle that we used in the first paragraph of the proof for the even case.

We excluded the case $k=3$ from the previous argument.
To complete the proof, Figure~\ref{color_GP12} shows the required coloring of $GP(12,5)$.
\end{proof}

\section{Subcubic trees}

We have not found a counterexample to Conjecture~\ref{r&b} among the subcubic graphs (of any connectivity).
Therefore, we pose the following slight strengthening of Conjecture~\ref{r&b}:

\begin{conj}
 Every subcubic graph on at least $7$ vertices posseses a red-blue vertex coloring such that the induced blue components are vertices or edges and
 every red component has minimum degree at least $1$ and contains no $3$-edge path.
\end{conj}

We confirm the statement for trees.

\begin{theorem}
 Every subcubic tree $T$ posseses a red-blue vertex coloring such that the induced blue components are vertices or edges and the
 red components are stars with $1,2$ or $3$ edges.
\end{theorem}

\begin{proof}
Let $r$ be an arbitrary vertex of $T$.
Let us draw $T$ as a planar tree rooted at $r$.
The planar embedding allows us to distinguish the left and right son of a vertex. 
The level of $r$ is $0$ and the neighbors of $r$ lie in level $1$ etc.
We color the vertices of $T$ according to a breadth-first search.
In each step, we color a new vertex $v$, that has degree 1 in the current colored subtree.
We denote the only neighbor of $v$ by $x$.
In certain cases, we recolor a few previously colored vertices in a small neighborhood of $v$.
Any other vertex keeps its color.
Our coloring algorithm works according to the rules listed below.
Each rule describes a step, when we transform a colored subtree $T_i$ to $T_{i+1}$ adding the next vertex of the breadth-first search.
At the end of each step, we have a coloring of $T_{i+1}$ that satisfies the conditions of the theorem.

Case 0.
We color $r$ blue.

Case 1. 
Assume the degree of $x$ is at most 1 in $T_i$ and $x$ is a singleton blue.
Now we color $v$ blue.

Case 2.
Assume the degree of $x$ is 1 in $T_i$ and $x$ is red.
Now we color $v$ blue.

Case 3.
Assume the degree of $x$ is 1 in $T_i$ and $x$ is blue and its parent is also blue.
Now we color $v$ red and change the color of $x$ to red.

Case 4.
Assume the degree of $x$ is 2 in $T_i$ and $x$ is red.
Now we set $v$ to be blue.

Case 5.
Assume the degree of $x$ is 2 in $T_i$ and $x$ is blue.
By the conditions, the left son $l$ of $x$ must be blue and the parent $p$ of $x$ must be red in $T_{i}$.
Let $g$ denote the parent of $p$ in $T_i$, if it exists.

Case 5a.
Assume the degree of $p$ is 2 in $T_i$.
Now we set $l,v,x$ to be red and $p$ to be blue.
This works, unless now $g$ would be a singleton red in $T_{i+1}$.
In this exceptional case, we keep the colors of the vertices in $T_i$ and only change the color of $x$ to red and color $v$ blue.
If $g$ does not exist, then $p$ is the root, and $p$ has another son $x'$.
Now $x'$ plays the role of $g$ in the previous argument and we make the same recoloring if necessary.

Case 5b.
Assume the degree of $p$ is 3 in $T_i$, $x$ is the right son of $p$ and the left son of $p$ is blue.
It implies that the parent $g$ of $p$ is also red.
Now we set $l,v,x$ to be red and color $p$ blue.
This yields a coloring satisfying the required conditions, unless now $g$ is a singleton red.
In this exceptional case, we do a different recoloring of $T_i$. 
We change $x$ to red and color $v$ blue.

Assume the degree of $p$ is 3 in $T_i$, $x$ is the right son of $p$ and the left son of $p$ was red in $T_i$.
Now we set $x$ to red and $l,v$ to blue.
Here we notice that the recoloring of $x$ might create a red $3$-star, but not a red path with 3 edges. 

Case 5c. 
Assume the degree of $p$ is 3 in $T_i$, $x$ is the left son of $p$ and the right son of $p$ is blue.
It implies that the parent $g$ of $p$ is also red.
Now we set $l,v,x$ to be red and color $p$ blue.
This yields a coloring satisfying the required conditions, unless now $g$ is a singleton red.
In this exceptional case, we do a different recoloring of $T_i$. 
We change $x$ to red and color $v$ blue.

Assume the degree of $p$ is 3 in $T_i$, $x$ is the left son of $p$ and the right son of $p$ is red.
Now we set $x$ to red and $l,v$ to blue.

Since we covered all cases, the algorithm terminates with a coloring of $T$ satisfying the conditions of the theorem.
\end{proof}

We remark that in many cases, the following simple idea works: 
we repeatedly color one level of vertices blue and the next two levels red.
The failure of this coloring happens at red leaves, if their parent's color is blue.
However, taking subgraphs does not keep the properties required by the theorem.
Therefore, it seems difficult to make this simple idea into a full proof.

\section*{Discussion}

Since each monochromatic connected subgraph is very small in Conjecture~\ref{r&b}, 
we might call such colorings {\em crumby}.
Conjecture~\ref{r&b} concerns crumby colorings of 3-connected cubic graphs. 
Every such graph posseses a perfect matching by Petersen's Theorem.
Naturally, one may ask whether these graphs admit a red-blue induced matching, that is, 
we color the vertices red-blue and each color class induces a matching.
This is clearly a special case of a crumby coloring.
There are two non-isomorphic graphs arising from $K_{3,3}$ by adding a handle.
One of them is the Wagner graph, which posseses a red-blue induced matching.
Let the other graph be $H$. We claim that $H$ does not have a red-blue matching.
Actually, $H$ has a unique perfect matching (up to symmetry), as shown in Figure~\ref{handle_color}.
However, the four matching edges are pairwise connected by an edge, therefore cannot be colored red-blue.
Notice that $H$ do have a crumby coloring, see Figure~\ref{handle_color}.

 \begin{figure}[!h]
 \begin{center}
  \includegraphics[width=0.7\linewidth]{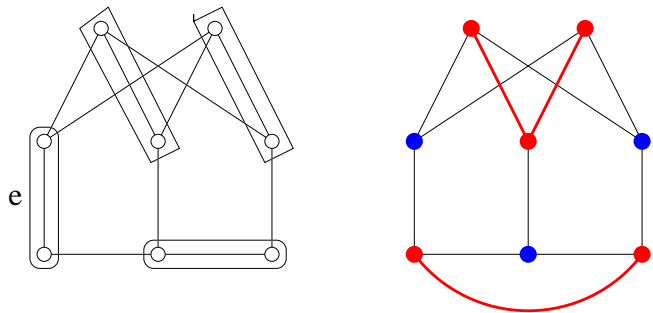}
  \caption{The matching on the left and a crumby coloring on the right.}
\label{handle_color}
 \end{center}
 \end{figure}

There are also infinitely many bipartite examples without a red-blue induced matching.
Let $G$ be a 2-connected bipartite cubic graph.
It contains a perfect matching $M$.
Suppose that $G$ also posseses a red-blue induced perfect matching.
Let $G(R,B)$ denote the following auxiliary bipartite multigraph: the vertices are the red matching edges in one class and the blue matching edges in the other class.
The edges are inherited from $G$.
That is, $(X,Y)$ is an edge in $G(R,B)$ if and only if $X$ is a red matching edge, $Y$ is a blue matching edge and 
there are vertices $x\in X$ and $y\in Y$ such that $(x,y)\in E(G)$.
Since $G$ was cubic, $G(R,B)$ is $4$-regular. Therefore, the number of red matching edges must equal the number of blue matching edges.
This yields the number of vertices in $G$ must be divisible by $4$.
However, there are infinitely many 3-connected, bipartite, cubic graphs on $4k+2$ vertices.
These graphs cannot have a red-blue induced matching by the previous argument.

Is this necessary condition also sufficient?

\medskip

\noindent {\it
  Does every $3$-connected, bipartite, cubic graph on $4k$ vertices posseses a red-blue induced perfect matching? 
 }

\medskip
 
 The answer to this question is negative for small values of $k$.
 Already the graph in Figure~\ref{handle_color} is a counterexample.
 There are five connected, cubic, bipartite graphs on 12 vertices.
 One of them is 2-connected, but not 3-connected.
 Among the 3-connected ones, there is 1 graph $G$ without a red-blue induced perfect matching.
 The reason is the following: $G$ contains $K_{2,3}$ as a subgraph. 
When we color these 5 vertices, at least 3 of them receive the same color. 
They either induce a structure larger than a matching edge, or they isolate a vertex of the other color.
%
%
 
  \begin{figure}[!h]
 \begin{center}
  \includegraphics[width=0.2\linewidth]{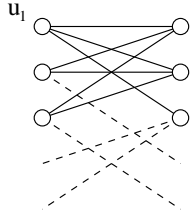}
  \caption{Each three possible pairing of $u_1$ leads to a contradiction.}
\label{gadget}
 \end{center}
 \end{figure}

 There are more examples on 16 vertices without a red-blue induced perfect matching. 
 However, there seems to be no local reason. 
 
 
 \begin{problem}
  Can we characterize the $2$-edge-connected cubic graphs, which posses a red-blue induced perfect matching? 
 \end{problem}

One might be tempted to stretch to the limits of Conjecture~\ref{r&b}. 
In our solutions, we used the red $3$-star.
Can this be avoided?

\begin{problem}
  Are there infinitely many subcubic graphs, which do not posses a red-blue coloring such that each blue component has $1$ or $2$ vertices and 
  every red component has $2$ or $3$ vertices? 
 \end{problem}
 


The following generalization of Conjecture~\ref{r&b} seems plausible.

\medskip

\noindent {\it Every graph $G$ with maximum degree $\Delta\le k$ has a red-blue coloring such that every blue component is either an isolated vertex or an edge and
 there is neither a red isolated vertex nor a red path with $k$ edges.} 
 
\medskip
 
However, the toroidal grids $C_5\times C_5$ and $C_5\times C_7$ do not posses the required red-blue coloring.
Probably the same holds for any $C_k\times C_l$ for odd numbers $k$ and $l$.

Therefore, it remains open whether the conjectured phenomenon in Conjecture~\ref{r&b} generalizes to higher degrees.

\end{document}